\documentclass[11pt]{amsart} 
\usepackage[margin=1.25in]{geometry}
\usepackage{amsmath,amsthm,amsfonts,amssymb,verbatim,eucal,bm,pifont}
\usepackage[all]{xy}
\usepackage[capitalise]{cleveref}
\usepackage{xcolor}
\usepackage{amssymb} 
\numberwithin{equation}{section}

\newtheorem{thm}[equation]{Theorem} 
\newtheorem{prop}[equation]{Proposition}
\newtheorem{lemma}[equation]{Lemma} 
\newtheorem{cor}[equation]{Corollary}
\newtheorem{example}[equation]{Example}
\newtheorem{remark}[equation]{Remark}
\newtheorem{definition}[equation]{Definition}

\DeclareMathOperator{\gr}{gr}

\DeclareMathOperator{\length}{\ell}

\newcommand{\NN}{\mathbb N}
\newcommand{\RR}{\mathbb R}
\newcommand{\FF}{\mathbb F}

\newcommand{\ld}{\lambda}

\renewcommand{\ker}{\mbox{\rm Ker\,}}

\newcommand{\ot}{\otimes}
\newcommand{\s}{\mathfrak{S}}

\newcommand{\cH}{\mathcal{H}}

\newcommand{\CC}{\mathbb{C}}  
\newcommand{\ZZ}{\mathbb{Z}}

\DeclareMathOperator{\codim}{codim}

\newcommand{\field}{\FF}
\begin{document}
\begin{abstract}
  We investigate deformations of 
  skew group algebras arising
  from the action of the symmetric group on polynomial rings
  over fields of arbitrary characteristic.
  Over 
  the real or 
  complex numbers,
   Lusztig's graded
  affine Hecke algebra
  and analogs are all isomorphic to
  Drinfeld Hecke algebras, 
  which include the symplectic reflection algebras and rational Cherednik algebras.  Over fields of prime characteristic, new deformations
  arise that capture both a 
  disruption of the group action and also a 
  disruption of the commutativity 
  relations defining the polynomial
  ring.  We classify
deformations for the symmetric group
acting via its natural (reducible)
reflection representation.
  \end{abstract}
\title[Symmetric groups in positive characteristic]
      {Drinfeld Hecke algebras for\\
symmetric groups in positive characteristic}

\date{November 18, 2021}
\author{N.\  Krawzik}
\address{Department of Mathematics and Statistics, Sam Houston State University,
Huntsville, Texas 77340, USA}
\email{krawzik@shsu.edu}
\author{A.V.\ Shepler}
\address{Department of Mathematics, University of North Texas,
Denton, Texas 76203, USA}
\email{ashepler@unt.edu}
\thanks{MSC-Numbers: 20C08, 20C20, 16E40}
\thanks{Second  author  partially  supported  by  Simons  Foundation Grant \#42953.}

\maketitle

\section{Introduction}

Deformations of skew group algebras constructed from finite groups
acting on polynomial rings
are used in representation
theory, combinatorics, and the study of orbifolds.
These deformations
include graded affine Hecke algebras,
Drinfeld Hecke algebras, rational Cherednik algebras, and symplectic reflection algebras.
The skew group algebra $S\#G$ arising from a group $G$
acting on an algebra $S$
by automorphisms
is the natural semidirect product algebra $S\rtimes G$.

 Lusztig~\cite{Lusztig88, Lusztig89}
defined deformations
of skew group algebras $S\#G$
for a Weyl
group $G$ acting
on its reflection representation $V= \RR^n$
and its associated polynomial ring
$S=S(V)=\RR[v_1, \ldots, v_n]$ in his investigations of
the representation theory
of groups of Lie type.
His algebras
alter
the relations 
$g\cdot s = g(s)\cdot g$ 
capturing
the group action but
preserve the commutativity 
relation $v_iv_j=v_jv_i$
defining the polynomial ring $S$.
These algebras are known
as {\em graded affine Hecke algebras}.  Around the same time,
Drinfeld~\cite{Drinfeld}
more broadly
considered an arbitrary
subgroup $G$ of $\text{GL}_n(\CC)$.
He defined
a deformation 
of $S\# G$,
sometimes called
the {\em Drinfeld Hecke algebra},
by instead
altering the commutation
relation $v_iv_j = v_jv_i$
defining the polynomial
ring $S$ (now defined over $\CC$) but
leaving the group action relation alone.
Drinfeld's type of deformation
was rediscovered by Etingof and Ginzburg~\cite{EtingofGinzburg} in the study
of orbifolds as 
{\em symplectic reflection algebras}
when the acting group
$G$
is symplectic, which include
{\em rational Cherednik algebras}
as a special case.
Collectively, we 
call all these deformations
{\em Drinfeld Hecke algebras}.

Over $\CC$,
every deformation
modeled on Lusztig's graded affine Hecke algebra is
isomorphic to a deformation
modeled on Drinfeld's Hecke algebra
(see~\cite{RamShepler}).
This result
holds for any finite group $G$ in the {\em nonmodular setting}, i.e., when 
the characteristic
of the underlying field $\FF$
is coprime to
the group order $|G|$
(see~\cite{SheplerWitherspoon2015}),
and in the present work
we extend this result to algebras
that deform {\em both} Lusztig and Drinfeld
type relations.
The theorem fails
in the {\em modular setting},
when $\text{char}(\FF)$
divides $|G|$, and
new types of deformations arise.
We begin here a concrete study of
these mixed deformations
using PBW conditions from~\cite{SheplerWitherspoon2015}.

We consider the symmetric group $\mathfrak{S}_n$
acting by permutations
of basis elements on a
vector space $V\cong \FF^n$
over a field $\FF$.
We include the case when 
$\text{char}(\FF)$ divides
the order of $\mathfrak{S}_n$. 
 We deform the natural
semi-direct product algebra 
$S\rtimes \mathfrak{S}_n$
for $S=S(V)\cong \FF[v_1, \ldots, v_n]$, the polynomial
ring,
by introducing relations
that disturb both the action of $\mathfrak{S}_n$
on $V$ and the commutativity of
the polynomial ring
$S$. 
We classify
the deformations arising in this way.

Recently, authors have 
considered 
the representation theory
 of similar deformations 
 over fields of positive characteristic
 defined by relations of Drinfeld type.
For example,
Bezrukavnikov, Finkelberg, and Ginzburg~\cite{BFG}, Devadas and Sun~\cite{DevadasSun},
Devadas and Sam~\cite{DevadasSam},
Cai and Kalinov~\cite{CaiKalinov},
and Lian~\cite{Lian}
all consider the Weyl group of type
$A_{n-1}$ (the irreducible reflection action 
of the symmetric group $\mathfrak{S}_n$)
acting on $V\oplus V^*$
to give rational Cherednik algebras.
Bellamy and Martino~\cite{BellamyMartino}
as well as Gordon~\cite{Gordon}
investigate
the action of
the symmetric group $\mathfrak{S}_n$
in the nonmodular setting,
when char$(\FF)$ and $|\mathfrak{S}_n|$
are coprime.
For related algebras built on the action
of the special linear group, general linear group,
and cyclic groups, see
Balagovi\'c and Chen~\cite{BalagovicChen,
BalagovicChen2} and  Latour~\cite{Latour}.
In contrast, the algebras classified here
are defined by relations of  both
Lusztig and Drinfeld type.

\begin{example}
Let $\field$ be a field of characteristic $p\geq 0$, $p\neq 2$, and let $G=\s_n$ act on $V=\field^n$ by permuting basis vectors
$v_1, \dots, v_n$ for $n> 2$.
Using \cref{PBWconditions} below, one may show that
the $\FF$-algebra
generated by $v$ in $V$
and the group algebra $\FF G$ 
 with  relations
\[\begin{aligned} 
v_iv_j-v_jv_i&=\sum_{k\neq i,j}
\big((i\ j\ k)-(j\ i\ k)\big)
\text{ in } \FF G \quad\text{ and }\quad
gv_i-\ g(v_i)g
=\big(g(i)-i\big)g
\text{ in } \FF G\\
\end{aligned}\]
is a PBW deformation of $\FF[v_1, v_2, ..., v_n]\rtimes G$.
\end{example}

\vspace{2ex}

{\bf Outline.}
We review Drinfeld Hecke algebras in Section 2 and PBW conditions
in Section~3.
In Section~4, we show that every
Drinfeld Hecke algebra in the nonmodular
setting is isomorphic to one
in which the skew group relation is not deformed.
We turn to the modular setting for the rest
of the article and consider the symmetric group in its permutation representation.
We examine PBW conditions in Sections 5 and 6.
We classify the Drinfeld Hecke algebras
in Section~7 and give these algebras explicitly
for dimensions $n=1,2,3$ in Section~8.
We conclude 
by investigating
the case of an invariant parameter in Section 9.

\vspace{2ex}

{\bf Notation.}
Throughout, we fix a field $\field$ 
of characteristic $p\geq 0$, $p\neq 2$,
and unadorned tensor symbols 
denote the tensor product over $\field$: 
$\otimes = \otimes_{\field}$.  We assume all $\field$-algebras are associative
with unity $1_{\FF}$
which is identified with the group identity $1_G$ in any group algebra $\FF G$.

\section{Drinfeld Hecke algebras}\label{sec:quadratic}

We consider a  finite group $G\subset \text{GL(V)}$
acting linearly on a finite-dimensional
vector space $V\cong \field^n$
over the field $\field$.
We may 
extend the action to one by automorphisms
on 
the symmetric algebra $S(V)\cong\FF[v_1, \dots, v_n]$  
for an $\FF$-basis
$v_1, \dots, v_n$ of $V$
and on the tensor algebra $T(V)=T_\FF(V)$ (i.e., the free associative $\field$-algebra
generated by $v_1, \ldots, v_n$).
We write the action of $G$ on any $\FF$-algebra $A$ with left superscripts,
$a\mapsto\ ^g a$ for $g$ in $G$,
$a$ in $A$, to avoid confusion
with the product $g\, a$ in any
algebra containing 
$\FF G$ and $A$.


\subsection*{Skew group algebra}
Recall that the {\bf skew group algebra} 
$R\#G
=R\rtimes G$
of a group $G$ acting by automorphisms
on an $\field$-algebra $R$ (e.g., $S(V)$ or $T(V)$)
is 
 the  $\FF$-algebra
generated by $R$ and $\field G$
with relations
$ g\, r = \,^{g} r \, g$
for 
$r$ in $R$, $g$ in $G$.
This natural semi-direct product algebra is also called the
{\em smash product}
or {\em crossed product algebra}.

\subsection*{Relations
given in terms of
parameter functions}
We consider an algebra
generated by $v$ in $V$
and $g$ in $G$ with
relations given in terms
of two parameter functions
$\ld$ and $\kappa$.  The parameter
$\ld$ measures the failure
of group elements
to merely act when they 
are interchanged with a vector $v$ in $V$,
and $\kappa$ measures
the failure of vectors in $V$ to commute.

Let
$\cH_{\ld,\kappa} $ 
be the associative $\field$-algebra generated 
by $v$ in $V$ together with the group algebra
$\field G$  with additional relations
\begin{equation}\label{relations}
\begin{aligned}
\text{}\ vw-wv=\kappa(v,w)
\ \ \text{ and }\ \
\text{}\ g v-   \,^g\! v g = \ld(g,v)
\quad\text{ for }\ 
v,w \in V,\ g\in G
\end{aligned}
\end{equation}
for a choice of
linear parameter functions $\lambda$ and $\kappa$, 
with $\kappa$ alternating,
$$
\kappa:V\ot V \rightarrow \field G,\quad
\lambda:\field G\ot    V \rightarrow \field G\, .
$$
For ease with notation,
we write $\kappa(v,w)$ for $\kappa(v\ot w)$
and $\ld(g,v)$ for $\ld(g\ot v)$ throughout.
Thus $\cH_{\lambda,\kappa}$
is a quotient of the free 
$\FF$-algebra generated
by $G$ and the basis $v_1,\ldots, v_n$ of $V$:
$$
\begin{aligned}
\cH_{\ld,\kappa} 
=
\FF\langle
v_1, \ldots, v_n, g: g\in G
\rangle
/
(
&gh - g\cdot_G h,\\
&vw-wv-\kappa(v,w),\\
&gv- \,^g\! v g - \ld(g,v):
v,w \in V,\ g,h \in G)\, .
\end{aligned}
{}_{}
$$
Note that $\cH_{0,0}=S(V)\# G$,
which is isomorphic
to $S(V)\ot \FF G$
as an $\FF$-vector space.

\vspace{2ex}

\subsection*{PBW property}
Recall that a filtered algebra satisfies
the PBW property
with respect to a set of generating
relations
when its homogeneous version 
coincides with its 
associated graded algebra.
We filter the algebra
$\cH_{\ld,\kappa}$ by degree
after assigning degree $1$ to
each $v$ in $V$ and degree $0$
to each $g$ in $G$.
The homogeneous version of
$\cH_{\ld,\kappa}$ is then
the skew group algebra
$S(V)\# G$, while
the associated 
graded algebra 
$\gr\cH_{\ld,\kappa}$ 
is a quotient of $S(V)\# G$
(e.g., see Li~\cite[Theorem~3.2]{Li2012}
or Braverman and Gaitsgory~\cite{BG}).
Thus we say $\cH_{\ld,\kappa}$ 
exhibits the {\em Poincar\'e-Birkhoff-Witt (PBW) property}
when
$$
\gr \cH_{\ld,\kappa} \cong S(V)\# G
$$
as graded algebras.
We call $\cH_{\ld,\kappa}$ a 
{\em Drinfeld Hecke algebra} 
in this case.
Note that every Drinfeld Hecke algebra 
has $\FF$-vector space basis
$\{v_1^{i_1} v_2^{i_2} \dotsm v_n^{i_n} g:
i_m \in \NN, g\in G\}$
for any $\FF$-vector space basis
$v_1,\ldots, v_n$ of $V$.

This terminology arises from the fact
that
Lusztig's graded versions of the affine Hecke algebras
(see~\cite{Lusztig88, Lusztig89})
are special cases of the PBW algebras
$\cH_{\ld, \kappa}$:
Lusztig's algebras arise
in the case that
$\kappa\equiv 0$ 
and $G$ is a finite Coxeter group.
(The parameter $\ld$ may be defined using a simple 
root system 
for $G$ and BGG operators).
Drinfeld's algebras~\cite{Drinfeld}
arise in the case that $\lambda\equiv 0$.
Examples 
when $\ld\equiv 0$
include
rational Cherednik algebras and symplectic reflection algebras
(see Etingof and Ginzburg~\cite{EtingofGinzburg}).


\section{Poincar\'e-Birkhoff-Witt property}
In this section,  
we recall necessary and sufficient
conditions
on the parameters
$\ld$ and $\kappa$
from~\cite{SheplerWitherspoon2015} for a filtered
quadratic algebra $\cH_{\ld,\kappa}$ to 
exhibit the PBW property.
Again, let $G\subset \text{GL}(V)$
be a finite group
acting on $V\cong \field^n$.
Let
$$\lambda_g: \field G\ot V\rightarrow \field
\quad\text{ and }\quad
\kappa_g: V\otimes V\rightarrow \field
$$ 
be the coefficient
functions for the parameters $\ld$ and $\kappa$, i.e., the
$\field$-linear maps for which 
$$
   \lambda(h,v)=\sum_{g\in G}\lambda_g(h,v) g
    \quad\text{ and }\quad \kappa(u,v) 
    =\sum_{g\in G} \kappa_g(u,v)g 
\quad\text{
for } g,h\in G,\
u,v\in V.
$$ 

\subsection*{Group action on parameters}
The group $G$ acts on the space of
parameter functions  $\ld$ and $\kappa$ in the usual way,
\begin{equation}
\label{groupactiononparameters}
(^h\kappa)(v,w)=
\ ^h\big(\kappa(\ ^{h^{-1}}\!v,\ ^{h^{-1}}\!w)\big),
\ (^h\ld)(g, v)
=\ ^h\big(\ld(h^{-1}gh, \ ^{h^{-1}}\! v)\big),
\end{equation}
using the action of $G$ 
on $\FF G$ by conjugation,
$g: h\mapsto g h g^{-1}$.

\subsection*{PBW conditions}
The second  PBW
condition
below
measures the failure of $\kappa$ to be $G$-invariant
while the first
shows that $\ld$ is determined
by its values on generators of $G$. 


\vspace{1ex}

\begin{thm}[\bf PBW Conditions]
\label{PBWconditions}
For any finite group $G\subset \text{GL}(V)$,
an algebra $\cH_{\ld,\kappa}$ exhibits the PBW property if and only if 
the following five conditions hold
for all $g,h$ in $G$, $u,v,w$ in $V$:
\begin{enumerate}
\item\label{cocycle21}
\,  \rule[0ex]{0ex}{2.5ex}
$\ld(gh,v)=\ld(g,\, ^hv) h + g\ld(h,v)$ in $\field G$.
\item\label{firstobstruction12}\ \ \rule[0ex]{0ex}{2.5ex}
$\kappa(\, ^gu,\, ^g v)g-g\kappa(u,v)
=
\ld\bigl(\ld(g,v),u\bigr)-\ld\bigl(\ld(g,u),v\bigr)$ in $\field G$.
\item\label{cocycle12}\ \ \rule[0ex]{0ex}{2.5ex}
$\ld_h(g,v)(\, ^hu-\, ^gu)=\ld_h(g,u)(\, ^hv-\, ^gv)
$ in $V$.
\item\label{firstobstruction03}
\ \ \rule[0ex]{0ex}{2.5ex}
$\kappa_g(u,v)(^gw-w)+\kappa_g(v,w)(^gu-u)+\kappa_g(w,u)(^gv-v)=0$ in $V$.
\item\label{mixedbracket}
\ \ \rule[0ex]{0ex}{2.5ex}
$\ld\bigl(\kappa(u,v),w\bigr)
+\ld\bigl(\kappa(v,w),u\bigr)+\ld\bigl(\kappa(w,u),v\bigr)=0$ in $\field G$.
\end{enumerate}
\end{thm}
We will refer
to the specific conditions in this theorem,
each with universal quantifiers,
as {\em PBW Conditions (1) through (5)}.
Notice that if $\cH_{\lambda, \kappa}$
exhibits the PBW property,
so does $\cH_{c\lambda, c^2 \kappa}$
for any $c$ in $\FF$.

\vspace{1ex}

We will need two corollaries from~\cite{SheplerWitherspoon2015}, the first on $\kappa$
and the second on $\lambda$.
We set
$
\ker\kappa_g
=
\{v\in V:\kappa_g(v,w)=0 
\text{ for all } w\in V\}\, 
$
and say that $\kappa$ is {\em supported}
on a subset $S\subset G$ when $\kappa_h\equiv 0$
unless $h$ lies in $S$.  Similarly, we say that
$\ld(g,*)$ is {\em supported} on $S\subset G$
when $\ld_h(g,v)=0$ for all $v$ in $V$ and $h$ not in $S$.
 Here, $V^g=\{v\in V: \, ^ g v=v\}$
  is the fixed point space.
\vspace{1ex}

\begin{cor}\label{PBWkappaconditions}
Let $G\subset \text{GL}(V)$
be a finite group and
assume $\cH_{\lambda, \kappa}$ 
exhibits the PBW property. 
For every $g$ in $G$,
if $\kappa_g \not\equiv 0$,
then either 
\begin{enumerate}
\item[(a)] $\codim V^g = 0$ (i.e., $g$ acts as the identity on $V$),
or
\item[(b)]
$\codim V^g = 1$ and $\kappa_g(v_1,v_2)=0$
for all $v_1,v_2$ in $V^g$,
or
\item[(c)]
$\codim V^g=2$ with
$\ker\kappa_g = V^g$.
\end{enumerate}
\end{cor}
\vspace{1ex}

Recall that a nonidentity element of $\text{GL}(V)$
is a {\em reflection} if it fixes a hyperplane of $V$
pointwise.  We see in the above corollary that
over fields of positive characteristic,
a parameter $\kappa$ defining an algebra with the PBW
property may be supported on reflections in addition to the identity $1_G$ and bireflections (elements whose
fixed point space has codimension $2$) 
in contrast
to the possible support in the nonmodular setting.
\vspace{1ex}

\begin{cor}\label{cor:lambda}
Let $G\subset \text{GL}(V)$
be a finite group
and say $\cH_{\ld,\kappa}$ exhibits the PBW property.
Then the following statements
hold for any $g$ in $G$.
\begin{enumerate}
\item
$\ld(1,*)$ is identically zero:
$\lambda(1,v)=0$ for all $v$ in $V$.
\item \label{lg-inverse} 
$\ld(g,*)$ determines $\ld(g^{-1}, *)$ by
  $
  g\lambda(g^{-1},v)= - \lambda(g, {}^{g^{-1}}v)g^{-1}\, .
  $
\item
$\ld(g,*)$ can be defined recursively:
For $j\geq 1$,
$
\ld(g^j,v)=\sum_{i=0}^{j-1} g^{j-1-i}\, \ld(g,\, ^{g^i}v)\, g^i\ .
$
\item
$\ld(g,*)$ is supported on $h$ in $G$ with $h^{-1}g$ either
a reflection or the identity on $V$. 
If $h^{-1}g$ is a reflection, then 
$\ld_h(g,v)=0$ for all $v$ on the reflecting hyperplane
$V^{h^{-1} g}$.
\item
If $V^g\neq V$,
$\lambda_1(g,v)=0$
unless $g$ is a reflection and $v\notin V^g$.
\end{enumerate}
\end{cor}

We give an example
in which the parameter
$\kappa: V\otimes V \rightarrow \FF G$
is {\em not} $G$-invariant.

\begin{example}
{\em
Let
$G=\s_n$ act on $V=\field^n$ for $n> 3$
by permuting basis 
elements $v_1, \ldots, v_n$,
i.e., $v_i\mapsto \ ^gv_i=v_{g(i)}$
for $g$ in $G$.
Fix two scalar parameters
$m,m'$ in $\FF$.
Define  $a_{ij}$ in $\FF$
for $1\leq i\neq j\leq n$
by
$m=a_{12}=a_{13}=-a_{21}=-a_{31}$,
$m'=a_{23}=-a_{32}$, and  $a_{ij}=0$ otherwise.
Then
the algebra defined by
\[
\begin{aligned} 
v_1v_2-v_2v_1
&=v_2v_3-v_3v_2
=v_3v_1-v_1v_3
=
\,
m^2\big((1\ 3\ 2)-(1\ 2\ 3)\big),
 \ 
v_iv_j-v_jv_i
=0
\text{ otherwise},
\\
gv_i-v_{g(i)}g
&=
\sum_{j\neq i}\ 
( a_{ij}- a_{g(i)\, g(j)} )
\
g(i\ j)
\quad
\text{ for $g\in G$, $1\leq i \leq n$}
\, 
\end{aligned}
\]
is a PBW deformation of $S(V)\#\field G$ 
(see \cref{classification} with $c=a_{123}$).
Notice (see \cref{equation:alpha beta def}
and \cref{definition of Hf})
that $a_{ij}=(1/4)\lambda_1\big((i\ j), 
v_i-v_j\big)$.
}
\end{example}

\section{Nonmodular Setting}
\label{nonmodular}
Before classifying algebras
in the modular setting,
we verify in this section
that every Drinfeld Hecke algebra
$\cH_{\lambda, \kappa}$ in the nonmodular setting 
is isomorphic to
one with $\lambda\equiv 0$.
We consider an arbitrary 
finite group $G\subset \text{GL}
(V)$ acting on $V\cong \field^n$
but assume $\text{char}(\field)\neq 2$
does not divide $|G|$ (e.g., $\text{char}(\field)=0$)
in this section only.
In the next theorem,
we extend a result of
~\cite{SheplerWitherspoon2015}
from the special case 
in which  one of the parameter functions
is zero to the case of more general parameters;
the result
of~\cite{SheplerWitherspoon2015}
strengthened a theorem
in~\cite{RamShepler} 
from the setting of Coxeter groups 
to arbitrary finite groups.

\begin{thm}\label{LusztigyIsDrinfeldy}
Say $G\subset \text{GL}(V)$ is a finite
group for
$V\cong \FF^n$ with $\text{char}(\FF)\neq 2$ coprime to 
$|G|$.
If an algebra
$\cH_{\lambda,\kappa'}$ exhibits the PBW property
for parameters 
$\lambda: \field G\ot V \rightarrow \field G$
and
$\kappa':V\ot V\rightarrow \field G$,
then there is an
algebra $\cH_{0, \kappa}$
for some
parameter 
$\kappa:V\ot V\rightarrow \field G$
with
$$
\cH_{0,\kappa}
\ \cong\ \cH_{\ld, \kappa'}\,
\quad\text{
as filtered algebras}.$$
\end{thm}
\begin{proof}
Define a conversion function
$\gamma:V\rightarrow \field G$ by
$$
\gamma(v)=\frac{1}{|G|}\sum_{a,b\,\in\, G} \lambda_{ab}(b,\, ^{b^{-1}}v) a\  
\
\text{ and write $\gamma=\sum_{a\in G} \gamma_a a$}
$$
for coefficient functions $\gamma_a:V\rightarrow \FF$.
Let 
$\kappa:V\ot V\rightarrow \field G$ 
be the parameter function 
$$
\kappa(u,v)=\gamma(u)\gamma(v)-\gamma(v)\gamma(u)
+\lambda(\gamma(u),v)-\lambda(\gamma(v),u)
+ \kappa'(u,v).
$$
Consider the
 associative $\field$-algebra $F$ generated by
$V$ and the algebra $\field G$, i.e.,
$$F=T_{\field}(\field G\oplus V)/(g\ot h-gh, 1_{\field}-1_G: g,h\in G)$$
(identifying each $g$ with $(g,0)$).
Define an algebra homomorphism
$f:F \rightarrow \cH_{\lambda,\kappa'}$ by
$$
f(v)= v + \gamma(v) \quad\text{and}\quad f(g) = g
\quad\text{ for all } v \in V \text{ and } g \in G\, .
$$
One may use the
PBW conditions 
for $\cH_{\lambda,\kappa'}$ 
to show that the
relations defining $\cH_{0,\kappa}$ lie
in the kernel of $f$,
as in the proof of~\cite[Theorem~4.1]{SheplerWitherspoon2015}.
Thus $f$ factors through an onto, filtered 
algebra homomorphism
$$
f:
\cH_{0,\kappa}
\twoheadrightarrow \cH_{\lambda,\kappa'}\, .
$$

The $m$-th filtered components of
$\cH_{\lambda,\kappa'}$
and 
$\cH_{0,\kappa}$
are both  spanned over $\FF$
by the monomials
$v_1^{a_1}\dotsm v_n^{a_n}g$
for
$g \in G$ and $a_i\in \NN$ with $\sum_i a_i \leq  m$, for a fixed basis $v_1,\ldots, v_n$ of $V$.
This spanning set is in fact
a basis for
$(\cH_{\lambda,\kappa'})_m$
by the PBW property,
and thus
$$
\dim_{\FF}
(\cH_{0,\kappa})_m
\leq 
\dim_{\FF}
(\cH_{\lambda,\kappa'})_m
\, .
$$
The map $f$
restricts to a surjective linear transformation
of finite-dimensional $\FF$-vector spaces
on each filtered piece, 
$$f:
(\cH_{0,\kappa})_m
\twoheadrightarrow 
(\cH_{\lambda,\kappa'})_m\, ,
$$
and hence is injective on each filtered piece.
(Indeed, for any
$v_1^{a_1}\dotsm v_n^{a_n}g$ of filtered
degree $m$ in the
PBW basis for $\cH_{\lambda,\kappa'}$,
the element 
$
(v_1-\gamma(v_1))^{a_1}\dotsm 
(v_n-\gamma(v_n))^{a_n}g
$
in
$\cH_{0,\kappa}$
is a preimage under $f$ and also
has filtered degree $m$.)
Thus
$f$ is an isomorphism
of filtered algebras.
Notice that
$f$
in turn induces an isomorphism of graded algebras,
$
\gr\cH_{0,\kappa}\cong \gr\cH_{\lambda,\kappa'}\cong S(V)\# G\, ,
$
and $\cH_{0,\kappa}$ also exhibits the PBW property.
\end{proof}

The special 
case of \cref{LusztigyIsDrinfeldy} when $\kappa'\equiv 0$ is from
\cite{SheplerWitherspoon2015}.
Note that \cref{LusztigyIsDrinfeldy} fails 
over fields of positive characteristic
as the next example  from~\cite{SheplerWitherspoon2015}
shows:
not every algebra $\cH_{\lambda,0}$
(modeled on Lusztig's graded affine Hecke algebra) 
is isomorphic to an algebra $\cH_{0,\kappa}$
(modeled on the Drinfeld Hecke algebra).

\vspace{1ex}

\begin{example}
\em{
Let $G\cong \ZZ/2\ZZ$
be generated by
$
g=\left(
\begin{smallmatrix} 
1&1\\0&1
\end{smallmatrix}
\right)
$
acting on $V=\field_2^2$ 
with respect to an ordered basis $v,w$.
Consider the $\FF$-algebra $\cH_{\lambda,\kappa'}$
generated by $V$ and $\field G$
with relations
$$
gv=vg,\ \  gw=vg+wg+1,\ \ 
vw-wv=g\, .
$$
Then $\cH_{\lambda,\kappa'}$ satisfies
the PBW property
but is not 
isomorphic to  $\cH_{0, \kappa}$ for any 
parameter $\kappa$.
(Here, $\lambda(g,v)=\lambda(1,v)
=\lambda(1,w)=0$, 
$\lambda(g,w)=1$, and
$\kappa'(u,v)=g$.)
}
\end{example}

\vspace{2ex}

\cref{LusztigyIsDrinfeldy}
and PBW Condition~(2)
(with $\lambda\equiv 0$) imply
the next observation.
\begin{cor}
Every Drinfeld Hecke algebra in the nonmodular setting arises from a
parameter $\kappa$ that is {\em invariant}, 
i.e., satisfying
$$
\kappa(\, ^gu,\, ^g  v)g = g\kappa(u,v)
\quad\text{ for }g\in G,\ u,v \in V.
$$

\end{cor}

\vspace{2ex}


\section{Deforming the group action}
In this section and the next,
we lay the framework for a complete classification of Drinfeld
Hecke algebras
for the symmetric group $\s_n$ acting 
by permutation matrices on $V\cong \FF^n$. 
In Sections $7$ and $8$,
we classify these algebras by giving the
parameters $\ld$ and  $\kappa$ such that $\cH_{\ld,\kappa}$ is PBW.
In this section we obtain the form of the parameter $\ld$, and in the next section 
we describe the parameter $\kappa$.
We assume $n> 2$ here and in Sections $6$ and
$7$ for ease with notation;
we give the PBW algebras
explicitly for $n=1,2,3$ in Section $8$.

We consider the action
of the symmetric group
by permutations.
Let $G=\s_n$ act on $V\cong\field^n$
by permuting basis 
elements $v_1, \ldots, v_n$
of $V$,
i.e., $v_i\mapsto \ ^gv_i=v_{g(i)}$
for $g$ in $G$.
We write $s_i=(i\ \, i+1)$  for the adjacent transpositions
generating $G$ for $1\leq i<n$ and set
$s_n=(n\ 1)$ for ease with later notation.
Recall that we assume
$2\neq\text{char}(\FF)\geq 0$.
Fix linear parameter functions
\[\kappa:V\otimes V\rightarrow \field G\ \ \text{ and }\ \ \ld:\field G \otimes V \rightarrow \field G,
\quad
\text{with $\kappa$ alternating.}
\]
\subsection*{Scalar parameters of freedom}
We show how each PBW algebra $\cH_{\lambda, \kappa}$ has parameter $\lambda$ determined by 
certain values on reflections.
Note
the group action in
\cref{groupactiononparameters}
induces the usual action
on
 $\lambda_1:G\otimes V \rightarrow \FF$
with
$$
(^{h^{-1}}\ld_1)
\big((i\ j), v_i\big)
=
\ld_1\big(h(i\ j)h^{-1}, \ ^{h}v_i\big)
=
\ld_1\big((h(i), h(j)), v_{h(i)}\big)\ . 
$$
We define scalars in $\FF$
for any linear parameter
function
$\lambda: \FF G\otimes V
\rightarrow \FF G$:
\begin{equation}
\begin{aligned}
\label{equation:alpha beta def}
\alpha_{ij}
&:=&
&\tfrac{1}{4}\ \lambda_1\big((i\ j), v_i-v_j\big),\\
\alpha_{ijk}
&:=&
&\alpha_{ij}\alpha_{jk}
+\alpha_{jk}\alpha_{ki}
+\alpha_{ki}\alpha_{ij}, \\
^{g}\alpha_{ij}
&:=&
&\tfrac{1}{4}\ \lambda_1\big(
(g(i)\ g(j)), v_{g(i)}-v_{g(j)}\big), \\
^{g}\alpha_{ijk}
&:=&
&\alpha_{g(i)g(j)}\alpha_{g(j)g(k)}+\alpha_{g(j)g(k)}\alpha_{g(k)g(i)}+\alpha_{g(k)g(i)}\alpha_{g(i)g(j)},\text{ and}\\
\beta_k
&:=&
&\tfrac{1}{2}\ \lambda_{s_k}\big(s_k, v_k-v_{k+1}\big), 
\end{aligned}\end{equation}
for $1\leq i,j,k\leq n$
with $i,j,k$ distinct.
We take indices on $\beta$ modulo $n$
throughout to more easily cyclically permute parameters in later results.
We will see that
if 
$\cH_{\lambda,\kappa}$
satisfies the PBW property, then $\beta_k
=\ld_{s_k}(s_k, v_k)$ for all $k$
(by \cref{lemma:opposite on reflections}).

\vspace{1ex}

\subsection*{Determination of $\lambda$}
We prove in this section
that every PBW algebra $\cH_{\lambda, \kappa}$ 
 has
 parameter $\ld$ determined by its values $\alpha_{ij}$ and 
 $\beta_k$:
\begin{thm}\label{ld def} 
A parameter
 $\lambda$
 satisfies PBW Conditions~(1) and~(3) if and only if 
\begin{equation}
\label{equation: definition of ld} 
\ld(g, v_i)\
=\ \sum_{k=0}^{g(i)-i+n-1} \beta_{i+k}\, g +\sum_{1\ \leq\ j\neq i\ \leq\ n}(\alpha_{ij}-\ ^{g}\alpha_{ij})\,
g(i\ j)
\
\end{equation}
for all $g \in G$
and $1\leq i\leq n$ with $\beta_1+\cdots+\beta_n=0$.
\end{thm}
We collect some necessary observations
before giving the proof of this theorem
at the end of this section.

\vspace{0ex}

\subsection*{Lemmas for the proof of \cref{ld def}}
Recall the {\em (absolute) reflection
  length} function $\length: \mathfrak{S}_n\rightarrow \ZZ_{\geq 0}$ on $\mathfrak{S}_n$
    which gives 
  the minimal number $\length(g)$ of
  transpositions in a factorization of $g$
  into transpositions.
  The following observation is well-known
for reflection groups  over $\RR$ (e.g., see \cite{Carter}, \cite{FosterGreenwood},
  \cite{SheplerWitherspoonAdvances})
and we include a proof for
  the symmetric group acting
  over arbitrary fields
  for the sake of completeness.

 \begin{lemma}
  \label{codimlemma}
  For $g,h$ in $\mathfrak{S}_n$,
  $\length(g)=\codim V^g$,
  and $V^g\cap V^h= V^{gh}$ 
  when
  $\length(g)+\length(h)=\length(gh)$. 
 \end{lemma}
 
\begin{proof}
For $\mathfrak{S}_n$
acting on $V_\RR=\RR^n$ by permutation
of basis vectors, 
$\length(g)=\codim V_\RR^g$.
 But $\codim V_\RR^g
  =  \codim V^g$ (just consider
  the decomposition of $g$ into disjoint
  cycles and take orbit sums for invariant
  vectors, for example). Hence
$  \length(g)+\length(h)=\length(gh)$ if and only if
  $
    \codim V^{g} 
  + \codim V^{h}
  =\codim V^{gh} 
$.
In this case, $V^g\cap V^h=V^{gh}$
since
   $V^g\cap V^h\subset V^{gh}$
implies that
$\codim V^g+\codim V^h - \codim(V^{gh}) \geq 
\codim V^g+\codim V^h - \codim(V^g\cap V^h) \geq 0$.
\end{proof}
Our next observation
follows from PBW Condition~(3) with $h=g(i\ j)$, $u=v_i, v=v_j$:
\begin{lemma}
\label{lemma:opposite on reflections} 
If the parameter function $\ld$ 
satisfies PBW Condition~(3), then
\[
\ld_{g(i\ j)}(g, v_i)=-\ld_{g(i\ j)}(g, v_j)=\tfrac{1}{2}\ld_{g(i\ j)}(g, v_i-v_j)
\quad\text{ for } i\neq j,\
g\in G.
\]
In particular, 
$\ld_1\big((i\ j), v_i\big)
=
-\ld_1\big((i\ j), v_j\big)
=
\tfrac{1}{2}
\ld_1\big((i\ j), v_i-v_j\big)$.
\end{lemma}

\begin{lemma}\label{gg fixed}
If the parameter function 
$\lambda$
satisfies PBW Condition~(1),  then $\lambda_c(c, v)=0$ 
for all $c$ in $G$ and $v$ in $V^c$,
  the fixed point space.
\end{lemma}
\begin{proof}
  We induct on the (absolute) reflection
  length $\ell(c)$ of $c$ using
  \cref{cor:lambda}~(1),
  which follows directly
  from PBW Condition~(1) 
(see the proof of~\cite[Cor.\ 3.3]{SheplerWitherspoon2015}).
  First suppose $c$ is a reflection itself 
  with $v \in V^c$. Then by PBW Condition~(1),
  \[ 
  0=\lambda(1, v)
  =
  \lambda(cc, v)=\lambda(c, \ ^cv)c+c\lambda(c, v)
  =\lambda(c, v)c+c\lambda(c, v).
  \]
 The coefficient of the identity group element on the right-hand side is $0=2\ld_c(c, v)$.
 
 Now suppose the claim holds for all group elements $g$ with $\ell(g)=k$ and that
  $\length(c)=k+1$. Then $c=ab$ for some $a$ 
 with $\length(a)=k$ and some transposition $b$.  As
 $\length(ab)=\length(a)+\length(b)$,
the vector $v$ lies in $V^{ab}=V^{a}\cap V^b$ by \cref{codimlemma}.
 By PBW Condition~(1),
 $
 \lambda(c, v)=\lambda(a, ^bv)b+a\lambda(b, v),$
 and the result follows from
 the induction hypothesis applied to the terms with $c$:
 \[
 \lambda_c(c, v)=\lambda_{a}(a, v)+\lambda_b(b, v)=0.
 \]
\end{proof}

\begin{lemma}
\label{gg depends on trans only}
 \label{betas sum to zero} 
  Say the parameter 
  function $\lambda$ satisfies PBW Condition~(1).
  Then 
  \begin{enumerate}
  \setlength{\itemindent}{-.25in}
  \item[(a)]
    For any $g\in G$ and any $i$, 
  $
  \lambda_g(g, v_i)
  =
  \lambda_{(i\ g(i))}
  \big((i\ g(i)), v_i\big)
  =
  -\ld_{(i\ g(i))}\big((i\ g(i)), v_{g(i)}\big)$.
  \item[(b)]
 For any $k$-cycle $(l_1\ l_2\ 
\cdots \ l_k)$
in $G$,
$$
\ld_{(l_1\ l_{2})}\big((l_1\ l_{2}), v_{l_1}\big)
+
\ld_{(l_2\ l_{3})}\big((l_2\ l_{3}), v_{l_2}\big)
+\cdots+
\ld_{(l_k\ l_{1})}\big((l_k\ l_{1}), v_{l_k}\big)
=
0\, .
  $$
  \item[(c)]
  For any $i<j$,
  $$
  \begin{aligned}
\ld_{(i\ j)}&\big((i\ j), v_{i}\big)
=
\ld_{s_i}\big(s_i, v_{i}\big)
+
\ld_{s_{i+1}}\big(s_{i+1}, v_{i+1}\big)
+\cdots+
\ld_{s_{j-1}}\big(s_{j-1}, v_{j-1}\big)\, 
\, .
\end{aligned}
$$
\end{enumerate}
 \end{lemma}
  \begin{proof}
  For (a), assume that $v_i \notin V^g$, else the claim
  follows from \cref{gg fixed}.
    Set $j=g(i)$ and
    $h=g(i\ j)$.
    Then $v_{j}\in V^{h}$ and PBW Condition~(1)
    implies that
        \[
    \begin{aligned} \lambda(g, v_i)
    =
    \lambda(h, ^{(i\ j)}v_i)(i\ j)
    +
    h\lambda\big((i\ j), v_i\big)
    = 
    \lambda(h, v_{j})(i\ j)+h\lambda\big((i\ j), v_i\big).
    \end{aligned}
    \]
    We isolate the coefficient of $g$ and apply Lemma~\ref{gg fixed} twice:
    \[\lambda_g(g, v_i)=\lambda_{h}(h, v_{j})
    +
    \lambda_{(i\ j)}
    \big((i\ j), v_i\big)
    =
    \lambda_{(i\ j)}\big((i\ j), v_i\big)
    =
   -\lambda_{(i\ j)}\big((i\ j), v_j\big).
    \]    
For (b), 
\cref{gg fixed} and (a) imply
that,
for $g=(l_1\ \cdots\ l_k)$,
\[
\begin{aligned} 
0=\ld_g(g, \sum_{i=1}^kv_{l_i})=\sum_{i=1}^k\ld_g(g, v_{l_i})
=
\sum_{i=1}^{k-1}
\ld_{(l_i\ l_{i+1})}
\big((l_i\ l_{i+1}), v_{l_i}\big) + \ld_{(l_k\ l_1)}
\big((l_k\ l_1), v_k\big).
\end{aligned}
\]
For (c), just use (b) with cycle $(i\ i+1\ \cdots\ j)$ and  (a).
\end{proof}

\begin{remark}{\em 
Note that
\cref{gg depends on trans only}
implies that there are at most $n-1$ choices determining the values $\ld_g(g, v)$ for $g$ in $G$ and $v$ in $V$ in a PBW
algebra $\cH_{\ld, \kappa}$.
Indeed, 
the values of $\ld_g(g, v)$ are determined
by the values $\ld_{(i\ j)}((i\ j), v_i)$ for $i<j$
by part (a), which are determined by the values
$\beta_k=\ld_{s_k}(s_k, v_k)$ for $1\leq k< n$
by part (c).
}
\end{remark}


\begin{lemma}
 If the parameter function $\lambda$
 satisfies 
 PBW Conditions~(1) and~(3), 
 then for all 
$g \in G$,
 $
  \ld(g, v_1+
  \cdots +v_n)=0
  $.
\end{lemma}
\begin{proof}
PBW Condition~(3) implies Corollary~\ref{cor:lambda}\ (4)
(see the proof of~\cite[Cor.\ 3.3]{SheplerWitherspoon2015}), hence
\[
\begin{aligned}
\sum_{i=1}^n\ld(g, v_i)
&=
\sum_{i=1}^n\ld_g(g, v_i)g+\sum_{i=1}^n\sum_{j\neq i}\ld_{g(i\ j)}(g, v_i)g(i\ j)\\
&=
\ld_g\big(g, \sum_{i=1}^n v_i\big)g
+\sum_{1 \leq i < j \leq n }\big(\ld_{g(i\ j)}(g, v_i)+\ld_{g(i\ j)}(g, v_j)\big)g(i\ j).\end{aligned}
\]
This vanishes by
\cref{lemma:opposite on reflections} 
and \cref{gg fixed}.
 \end{proof}

\begin{lemma}\label{betas}
For the parameter function $\ld$ satisfying
PBW Conditions~(1) and~(3),
\begin{equation}
    \label{SumDefinition}
    \beta_1+\cdots +\beta_n=0
    \quad\text{ and }\quad
\lambda_g(g, v_i)=
\sum_{k=0}^{g(i)-i+n-1}\beta_{i+k}
\quad\text{ for all }
g \in G,\ 1\leq i\leq n\, .
\end{equation}
\end{lemma}
\begin{proof}
\cref{betas sum to zero}
with the cycle $(1\ 2\ \cdots\ n)$ 
implies that $\sum_{1\leq j\leq n} \beta_j =0$.
For $i<g(i)$,
\[\begin{aligned}
\sum_{k=0}^{g(i)-i+n-1}
\beta_{i+k}
=
\sum_{k=i}^{g(i)-1}\beta_{k}
+ \sum_{k=g(i)}^{g(i)+n-1}
\beta_k
=
\sum_{k=i}^{g(i)-1}\ld_{s_k}(s_k, v_{k}),
\end{aligned}
\]
which is just 
$
\ld_{(i\ g(i))}\big((i\ g(i)), v_{i}\big)
=\ld_g(g, v_i)$
by \cref{betas sum to zero}.
Similarly, for $g(i)<i$, 
\[\begin{aligned}
\sum_{k=0}^{g(i)-i+n-1}
\beta_{i+k}
= \beta_i+\beta_{i+1}+\cdots+ \beta_n+\beta_1+\cdots +\beta_{g(i)-1}
=
-( \beta_{g(i)} + \cdots + \beta_{i-1}),
\end{aligned}
\]
which again is just
$
-\ld_{(g(i)\ i)}\big((g(i)\ i), v_{g(i)}\big)
=\ld_g(g, v_i)$
by \cref{betas sum to zero}.
\end{proof}

\vspace{1ex}

\begin{cor}\label{ld zero for fixed v}
If the parameter $\ld$ satisfies (\ref{SumDefinition}),
then for any $k$-cycle 
$(l_1\ \cdots\ l_k)$ in $G$,
$$
\lambda_{(l_1\ l_2)}\big((l_1\ l_2), v_{l_1}\big) 
+  \lambda_{(l_{2}\ l_3)}\big((l_2\ l_3), v_2\big) + \cdots 
+ \lambda_{(l_k\ l_1)}\big((l_k\ l_1), v_{l_k}\big)
=
0\, .
$$\end{cor}
\begin{proof}
Since (\ref{SumDefinition}) implies that 
$$\ld_{(i\ j)}((i\ j), v_i)=\beta_i+\cdots +\beta_n+\beta_1+\cdots +\beta_{j-1} \quad\text{for } 1\leq i\neq j\leq n
$$
(note that the sum stops at $\beta_n$ when $j=1$
as $\beta_n=\beta_0$), the sum 
$\sum_{i=1}^k \ld_{(l_i\ l_{i+1})}((l_i\ l_{i+1}), v_{l_i})$ 
is a multiple of $\beta_1+\cdots +\beta_n$, which is zero by the first equality in (\ref{SumDefinition}).
 \end{proof}

\vspace{2ex}

 \subsection{Proof of \cref{ld def}}
We now have the tools to show that every PBW algebra $\cH_{\lambda, \kappa}$ 
 has
 parameter $\ld$ determined by the values $\alpha_{ij}$ and 
 $\beta_k$. 
 
\begin{proof}[Proof of \cref{ld def}]
We show PBW Conditions (1) and (3) are equivalent to these three facts
for all $g$ in $G$ and all $i$:
\begin{itemize}
    \item[(a)] $\lambda_{g(i\ j)}(g, v_i)
    =
    \alpha_{ij}-\ ^{g}\alpha_{ij}$
    for all $j\neq i$,
    \item[(b)] $\beta_1+\cdots+ \beta_n=0$ and
    $\ld_g(g, v_i)
    =\sum_{k=0 }^{g(i)-i+n-1 }\beta_k$, and
    \item[(c)]
    $\ld(g, v_i)$ is  supported on $g$ 
    and all $g(i\ j)$ for $j\neq i$.
    \rule{0ex}{2.25ex} 
\end{itemize}
Assume PBW Conditions~(1) and~(3) 
both hold.
\cref{betas} implies (b).
PBW Condition~(3) implies part (4) of
Corollary~\ref{cor:lambda} 
(see the proof of~\cite[Cor.\ 3.3]{SheplerWitherspoon2015}), 
implying (c).

We induct on the
(absolute) reflection length $\length(g)$
of $g$ in $G=\s_n$
to verify (a).
First, if $g=1$, then both sides of $(a)$ vanish
by \cref{cor:lambda}(1).
Now suppose
$g$ is a transposition
fixing $i$ and $j$.
Then $\alpha_{ij} = \ ^{g}\alpha_{ij}$ and the right-hand side of (a) is zero;
on the other hand,
 PBW Condition~(1) implies that
  \[
    \lambda\big((i\ j), v_i\big)g
    +(i\ j)\lambda(g, v_i)
    =
    \lambda(g, v_j)(i\ j)+g\lambda\big((i\ j), v_i\big)
    \]
    and we 
  apply Lemma ~\ref{lemma:opposite on reflections}
   to the coefficient of $g$,
    \[
    \ld_1\big((i\ j), v_i\big)
    +
    \ld_{g(i\ j)}(g, v_i)
    =
    \lambda_{g(i\ j)}(g, v_j)+
    \ld_1\big((i\ j),v_i\big),
    \]
 to
   see the left-hand side of (a) is zero.
Now suppose instead $g=(i\ k)$ for some $k\neq i$. Then  $g=(i\ j\ k)(i\ j)=(j\ k)(i\ j\ k)$,
and we use PBW Condition~(1) to write $\lambda((i\ k), v_i-v_j)$ in two ways and match the coefficients of $(i\ k)(i\ j)$:
\[
\begin{aligned} 
\lambda_{(i\ k)(i\ j)}
\big((i\ k), v_i-v_j\big)
&=
\lambda_{(i\ j\ k)(i\ j)}
\big( (i\ j\ k), v_j-v_i)+\ld_1((i\ j), v_i-v_j\big)
\quad\text{ and }
\\
\lambda_{(i\ k)(i\ j)}
\big((i\ k), v_i-v_j)
&=
\ld_1\big((j\ k), v_j-v_k\big)
+
\lambda_{(i\ j\ k)(i\ j)}
\big((i\ j\ k), v_i-v_j\big);
\end{aligned}
\]
we conclude $2\ld_{(i\ k)(i\ j)}((i\ k), v_i-v_j)
=4(\alpha_{ij}-\,
^{g}\alpha_{ij})$ and
\cref{lemma:opposite on reflections}
implies (a).

To show the induction step, 
fix some $g$ in $G$,
and assume the result holds for all group elements
with smaller (absolute) reflection length.
Write
$g=g_1 g_2$
for some $g_1, g_2$ in $\mathfrak{S}_n$ with
$0<\length(g_1), \length(g_2) < \length(g)$.
PBW Condition~(1) implies that
\[
\lambda(g, v_i-v_j)
= 
\lambda(g_1g_2, v_i-v_j)
 =
\lambda\big(g_1,\, ^{g_2}(v_i-v_j)\big)
g_2+g_1\lambda(g_2, v_i-v_j),
\] 
and we equate
the coefficients of 
$g(i\ j)
=
g_1 
\big(g_2(i)\ g_2(j)\big)g_2$ 
to obtain (a):
\[\begin{aligned}
2\lambda_{g(i\ j)}(g, v_i)
&=
\lambda_{g(i\ j)}(g, v_i-v_j)
= 
\lambda_{g_1(g_2(i)\ g_2(j))}(g_1, v_{g_2(i)}-v_{g_2(j)})
+
\lambda_{g_2(i\ j)}(g_2, v_i-v_j)
\\
&= 2(^{g_2}\alpha_{ij}-\ ^{g_1g_2}\alpha_{ij})
+2(\alpha_{ij}-\ ^{g_2}\alpha_{ij})
=
2(\alpha_{ij}-\ ^{g}\alpha_{ij})\, .
\end{aligned}
\]

  To prove the converse, we assume (a), (b), and (c) hold and first show PBW Condition~(1). 
  Note that (b) implies that for all $g\in G$ and for all $v_i \in V$,  $\ld_g(g,v_i)$ coincides
  with $\ld_{(i\ g(i))}((i\ g(i)), v_i)$. 
  The right-hand side of PBW Condition~(1)
  at $v=v_i$ is
  \begin{equation}\label{4terms}
  \begin{aligned}
  \ld_{(h(i)\  gh(i))}
  \big((h(i)\ gh(i) ), v_{h(i)}\big)\, gh
 \ &+
  \sum_{j:\, h(j)\neq h(i)}(\alpha_{h(i)h(j)}-\ ^g\alpha_{h(i)h(j)})
  \, gh(i\ j)h^{-1}h
  \\
   \ +\,  \ld_{(i\ h(i))}
  \big((i\ h(i)), v_i\big)
  \, gh 
  \ &+\ \ \ 
  \sum_{j: j\neq i}(\alpha_{ij}-\ ^h\alpha_{ij})\, gh(i\ j).
  \end{aligned}
  \end{equation}
We apply 
  \cref{ld zero for fixed v}
to the 3-cycle $(i\ h(i)\ gh(i))$ and the $2$-cycle
$(i\ gh(i))$, as well as (b), to simplify the $gh$ terms
and obtain
  \[\begin{aligned}
  - 
  \ld_{(i\ gh(i))}\big((i\ gh(i)), v_{gh(i)}\big) \, gh
     &=
   \ld_{(gh(i)\ i)}
   \big((gh(i)\ i), v_i\big)\, 
   gh\,
   = \ld_{gh}
   \big(gh, v_i\big)\, 
   gh\,
   .
   \end{aligned}
\]
The $gh(i\ j)$ terms combine as
$\sum_{j: i\neq j}\ld_{gh(i\ j)}(gh, v_i)\, gh(i\ j)$.
 Thus~\cref{4terms} is just 
 $\ld(gh, v_i)$.
 
  The only nontrivial case in showing PBW Condition~(3) is when $h=g(i\ j)$ with vectors $\{u,v\}=\{v_i,v_j\}$. To confirm that
  \[\lambda_{g(i\ j)}
  (g, v_i)
  (v_{g(i)}-v_{g(j)})=\lambda_{g(i\ j)}(g, v_j)(v_{g(j)}-v_{g(i)}),\]
   apply (a) to each side
  and note that
$      (\alpha_{ij}-\ ^g\alpha_{ij})(v_{g(i)}-v_{g(j)})
      =(\alpha_{ji}-\ ^g\alpha_{ji})(v_{g(j)}-v_{g(i)})$.
\end{proof}

\section{Deforming commutativity}
Again we consider $G=\mathfrak{S}_n$ acting by permutation
of basis vectors $v_1,\ldots, v_n$ of $V\cong\FF^n$. 
Working toward the classification of the Drinfeld Hecke algebras
$\cH_{\lambda,\kappa}$ in Section 7,
we described the form of $\ld$ in a PBW deformation in the last section and we describe the form of $\kappa$
here.
We again take $n> 2$ in this section
(see Section $8$ for $n=1,2,3$)
and 
fix linear parameter functions
\[\kappa:V\otimes V\rightarrow \field G\ \ \text{ and }\ \ \ld:\field G \otimes V \rightarrow \field G,
\quad
\text{with $\kappa$ alternating.}
\]
We use the PBW Conditions (1)--(5) of \cref{PBWconditions}.
The next lemma measures the failure of $\kappa$ to be invariant and follows
from \cref{ld def}. 

\vspace{1ex}

\begin{lemma}
\label{kappa invariance} Assume
PBW Conditions~(1) and~(3) hold
for the parameter function $\lambda$.
Then PBW Condition~(2)  is equivalent to 
\[
\kappa(\, ^gv_i,\, ^gv_j)g-g\kappa(v_i, v_j)
=
\sum_{k \neq i,j}
(^g
\alpha_{ijk}-\alpha_{ijk})
\, g\, \big((i\ j\ k)-(i\ k\ j)
\big)
\quad
\text{for all $g\in G$, $i\neq j$}.
\]
\end{lemma}
\begin{proof}
We rewrite the right-hand side of PBW Condition~(2) with $v_i, v_j$ in place
of $u, v$
using \cref{ld def}. \cref{betas sum to zero} 
implies all terms vanish except
those with 
group elements $g(i\ j\ k)$ and $g(i\ k\ j)$.
The coefficient of $g(i\ j\ k)$
is
\[
        \sum_{k\neq i,j}\Big((\alpha_{jk}-\ ^g\alpha_{jk})(\alpha_{ik}-\ ^g\alpha_{ij})-(\alpha_{ij}-\ ^g\alpha_{ij})(\alpha_{jk}-\ ^g\alpha_{ik})-(\alpha_{ik}-\ ^g\alpha_{ik})(\alpha_{ji}-\ ^g\alpha_{jk})\Big)
\]
 which simplifies to
 $
 \, ^g\alpha_{ijk}-\alpha_{ijk}
 $
 since $\alpha_{ij}=-\alpha_{ji}$, etc.
 Likewise,
 the coefficient 
   \[
  \sum_{k\neq i,j}\Big( (\alpha_{ji}-\ ^g\alpha_{ji})(\alpha_{ik}-\ ^g\alpha_{jk})-(\alpha_{jk}-\ ^g\alpha_{jk})(\alpha_{ij}-\ ^g\alpha_{ik})-(\alpha_{ik}-\ ^g\alpha_{ik})(\alpha_{jk}-\ ^g\alpha_{ji})\Big)
  \]
of  $g(i\ k\ j)$
simplifies to
$-(^g\alpha_{ijk}-\alpha_{ijk}) $
and we obtain the
right-hand side of the
equation in the statement.
\end{proof}

\vspace{1ex}

\begin{lemma}\label{KappaEqualities}
If $\cH_{\ld,\kappa}$ satisfies
the PBW property, then
for all distinct $i,j,k$, 
        \[
        \kappa_{(i\ j\ k)}(v_i, v_j)
        =
        \kappa_{(i\ j\ k)}(v_j, v_k)
        =
        \kappa_{(i\ j\ k)}(v_k, v_i)
        =
        \kappa_{(i\ k\ j)}(v_i, v_k)
        =
        -\kappa_{(i\ k\ j)}(v_k, v_i)\, .
        \]
\end{lemma}
\begin{proof}
We use Theorem~\ref{PBWconditions}.  PBW Condition~(4)
with $g=(i\ j\ k)$ implies that
  \[
  \kappa_{(i\ j\ k)}(v_i, v_j)(v_i-v_k)+\kappa_{(i\ j\ k)}(v_j, v_k)(v_j-v_i)+\kappa_{(i\ j\ k)}(v_k, v_i)(v_k-v_j)=0,
  \, 
  \]
giving the first two equalities.
Set
 $g=(i\ j)$ 
 in
 Lemma~\ref{kappa invariance} (whence $^g\alpha_{ijk}=\alpha_{ijk}$ and the right-hand side vanishes), and consider the coefficient of $(j\ k)$ to deduce the third equality. For the final equality, recall
 that $\kappa$ is alternating. 
 \end{proof}

\vspace{1ex}

\begin{prop}\label{kappa 3-cycles}  If $\cH_{\ld,\kappa}$ satisfies
the PBW property,
then $\kappa$ is supported on 3-cycles, and  
\[
\kappa(v_i, v_j)
=
\sum_{k\neq i,j}\kappa_{(i\ j\ k)}(v_i, v_j)\big(
(i\ j\ k)-(i\ k\ j)\big)\, 
\quad\text{ for } i\neq j\, .
\] 
\end{prop}
\begin{proof} 
We first show that $\kappa$ is supported on $3$-cycles using \cref{PBWconditions} and \cref{kappa invariance}. 
By Corollary~\ref{PBWkappaconditions},
$\kappa_g\equiv 0$ unless
$g$ is the identity, a transposition,
the product of two disjoint transpositions,
or a $3$-cycle.
Assume some $\kappa_g(v_i, v_j)\neq 0$
(so $i\neq j$).
Corollary~\ref{PBWkappaconditions}
in fact implies that 
$g$ must be 
$$
1_G,\ (i\ j),\ (j\ k),\ 
(i\ j)(k\ l),\ (i\ k)(j\ l),\  
(i\ j \ k), \text{ or }
(i\ k \ j)
\quad\text{ for some $k\neq l$ and $k,l\not\in\{i, j\}$}.
$$
Let $g=(i\ j)$ in \cref{kappa invariance}
and equate the coefficients of
$1_G$ to conclude that
\[\kappa_{(i\ j)}(v_j, v_i)-\kappa_{(i\ j)}(v_i, v_j)=0,\]
which implies $\kappa_{(i\ j)}(v_i, v_j)=0$
as $\kappa$ is alternating and $\text{char}(\FF)\neq 2$.
Similarly, we
equate the coefficients of $(i\ j)$ to deduce that 
$\kappa_1(v_i, v_j)=0$
and the coefficients of $(k\ l)$ 
to deduce that $\kappa_{(i\ j)(k\ l)}(v_i, v_j)=0$.  
Likewise,
set $g=(i\ j)(k\ l)$ and equate coefficients of $(i\ l)(j\ k)$ to see that  $\kappa_{(i\ k)(j\ l)}(v_i, v_j)=0$.
To verify that $\kappa_{(j\ k)}(v_i, v_j)=0$, 
on one hand we
notice that 
$\kappa_{(j\ k)}(v_i, v_j)=\kappa_{(j\ k)}(v_i, v_k)$
after setting $g=(j\ k)$ and equating the coefficients 
of $1_G$,
and 
on the other hand we notice that 
$\kappa_{(j\ k)}(v_i, v_j)=-\kappa_{(j\ k)}(v_i, v_k)$ 
by
PBW Condition~(4) with $g=(j\ k)$ and vectors $v_i, v_j$, and $v_k$.
Hence by
Lemma~\ref{KappaEqualities},
      \[\begin{aligned}
      \kappa(v_i, v_j)
      &=
      \sum_{k\neq i,j}\kappa_{(i\ j\ k)}(v_i, v_j)(i\ j\ k)+\kappa_{(i\ k\ j)}(v_i, v_j)(i\ k\ j)\\
      &=
      \sum_{k\neq i,j}\kappa_{(i\ j\ k)}(v_i, v_j)
      \big((i\ j\ k)-(i\ k\ j)\big).\end{aligned}
\]
\end{proof}
The next proposition gives an explicit formula for $\kappa(v_i,v_j)$ when $\cH_{\ld,\kappa}$ is PBW.

\vspace{1ex}

\begin{prop}
\label{kappa def} 
If $\cH_{\ld,\kappa}$ satisfies 
the PBW property, 
then 
\[
\kappa(v_i,v_j)=\sum_{k\neq i,j}\big(\alpha_{ijk}+\kappa_{(1\ 2\ 3)}(v_1, v_2)-\alpha_{123}\big)\big((i\ j\ k)-(i\ k\ j)\big)
\ \text{ for }
i\neq j.
\]
\end{prop}
\begin{proof}
By \cref{PBWconditions}, we may 
use \cref{kappa 3-cycles} 
and Lemma~\ref{kappa invariance} to write $g\kappa(v_i, v_j)$
in two ways
and then equate
the coefficients of $g(i\ j\ k)$
for distinct $i, j, k$ in $\{1, \ldots, n\}$:
\[
\kappa_{(i\ j\ k)}(v_i, v_j)=\kappa_{g(i\ j\ k)g^{-1}}(v_{g(i)},v_{g(j)})-
\ ^g\alpha_{ijk}+\alpha_{ijk}.\]
In particular, for $g=(i\ 1)(2\ j)(3\ k)$, 
$\kappa_{(i\ j\ k)}(v_i, v_j) = \kappa_{(1\ 2\ 3)}(v_1, v_2)-\alpha_{123}+\alpha_{ijk},
$
and \cref{kappa 3-cycles} implies the result.
\end{proof}

\vspace{1ex}

\begin{cor}\label{kappa def implies equalities}
If $\kappa$ is defined as in \cref{kappa def},
then
for all distinct $i,j,k$, 
        \[
        \kappa_{(i\ j\ k)}(v_i, v_j)
        =
        \kappa_{(i\ j\ k)}(v_j, v_k)
        =
        \kappa_{(i\ j\ k)}(v_k, v_i)
        =
        \kappa_{(i\ k\ j)}(v_i, v_k)=-\kappa_{(i\ k\ j)}(v_k, v_i)\, .
        \]
        \end{cor}
\section{Classification}

We now give the classification of
Drinfeld Hecke algebras for the
group $G =\s_n$ acting on $V\cong\FF^n$
by permuting basis vectors $v_1,\ldots, v_n$
of $V$.
Recall that $2\neq  \text{char}(\FF)\geq 0$.
We assume $n> 2$ in this section
for ease with notation;
see Section $8$ for the cases $n=1,2,3$.
We show that every Drinfeld Hecke algebra
has relations of a particular form
based on parameters $a_{ij}, b_k$ and $c$
in $\FF$.

\begin{definition}
\label{definition of Hf}
For any ordered tuple of scalars 
in $\FF$,
$$\mu=(a_{ij}, b_{k}, c:
\ 1\leq i < j\leq n,\ 
1\leq k<n ),
$$
let
$\cH_{\mu}$ be the $\FF$-algebra generated
by $V$ and $\FF G$ with relations
\begin{align}
gv_i- \, ^{g}v_ig 
&=
\hspace{-1ex}
\sum_{k=0}^{g(i)-i+n-1}
b_{i+k}\, g
+ 
\sum_{j \neq i}
(a_{ij}-\, a_{g(i)\, g(j)})
\ g(i\ j)
 \text{ for } g\in G, 1\leq i\leq n,
\label{equation: classify lambda}
\\
v_i v_j-v_jv_i
&=
\sum_{k\neq  i,j}(c-a_{123}+a_{ijk})
\big((i\ j\ k)-(i\ k\ j)\big)
 \quad\text{ for } 1\leq i\neq j\leq n
\label{equation: classify kappa}
\end{align}
where\ $a_{ijk}=
a_{ij}a_{jk}+
a_{jk}a_{ki}+
a_{ki}a_{ij}$,
$a_{ji}=-a_{ij}$,
and $b_n=-(b_1+\cdots + b_{n-1})$
with indices on $b$ taken
modulo $n$.
%
\end{definition}

We show that the above algebras $\cH_{\mu}$
make up the complete set 
of Drinfeld Hecke algebras:

\begin{thm}
\label{classification}
For
any Drinfeld Hecke algebra
$\mathcal{H}_{\lambda,\kappa}$
for $G=\mathfrak{S}_n$ acting on $V\cong\FF ^n$ by permutations,
there is an ordered tuple 
$\mu=(a_{ij}, b_{k}, c:1\leq i<j\leq n,1\leq k<n)$ of scalars 
so that 
$\mathcal{H}_{\lambda,\kappa}
=\cH_{\mu}$.
Conversely, for any choice $\mu$ of scalars,
$\mathcal{H}_{\mu}$ is a Drinfeld Hecke algebra.
\end{thm}
\begin{proof}
First assume $\cH_{\ld, \kappa}$ 
    satisfies the PBW property
    for some parameters
    $\ld$ and $\kappa$.
    Let $\cH_{\mu}$ 
    be the algebra of
    \cref{definition of Hf}
    for $\mu=(a_{ij}, b_k, c)$  defined 
    by
$$
\begin{aligned}
a_{ij}
&=&&
\tfrac{1}{4}\lambda_1\big((i\ j), v_i-v_j\big)
 &&  \ \  \text{ for } 1\leq i<j\leq n,
\\
b_k
&=&&
\tfrac{1}{2}\lambda_{(k\ k+1)}\big((k\ k+1), v_k-v_{k+1}\big)
&&\ \ \text{ for } 1\leq k<n,
\\
c
&=&&
\kappa_{(1\ 2\ 3)}(v_1, v_2),
&&\ \ \text{ and }
\\
a_{ijk}
&=&&
a_{ij}a_{jk}+a_{jk}a_{ki}+a_{ki}a_{ij}
.
\end{aligned}
$$
Using ~\cref{PBWconditions},
 we replace $\alpha_{ij}$ and $\beta_k$ with $a_{ij}$ and $b_k$, respectively, in ~\cref{ld def}
 to see that $\ld$ 
 defines the right-hand side of
 relation
 \cref{equation: classify lambda}.
Lemma~\ref{kappa invariance}
implies that
 \[\kappa_{g(i\ j\ k)g^{-1}}(v_{g(i)}, v_{g(j)})
 -\kappa_{(i\ j\ k)}(v_i, v_j) =\ ^ga_{ijk}-a_{ijk} 
 \quad\text{ for all }
 g\in G. 
 \] 
 Let $g=(1\ i)(2\ j)(3\ k)$ to see that
 \[
 \kappa_{(1\ 2\ 3)}(v_1, v_2)-\kappa_{(i\ j\ k)}(v_i, v_j)
 =
 a_{123}-a_{ijk}.
 \]
 \cref{kappa 3-cycles}
 then implies that $\kappa$
 gives the right-hand side of 
 \cref{equation: classify kappa}.
Thus
 $\cH_{\ld, \kappa}=\cH_{\mu}$.

Conversely, fix an algebra $\cH_{\mu}$
  and
 set $\ld(g,v_i)$ equal to the right-hand side of \cref{equation: classify lambda} and 
 set $\kappa(v_i, v_j)$ equal to the right-hand side of \cref{equation: classify kappa} for all $1\leq i\neq j\leq n$
 and $g$ in $G$.
 Extend $\ld$ and $\kappa$  to linear parameter functions 
 $\lambda: \FF G\otimes V \rightarrow \FF G$,
 $\kappa: V\otimes V
 \rightarrow \FF G$
 so that $\cH_{\mu}
 =\cH_{\lambda,\kappa}$.
 
 We show that $\cH_{\ld,\kappa}$
 is a Drinfeld Hecke algebra
 by verifying the five
 PBW Conditions of \cref{PBWconditions}. Results from previous sections
 (with $\alpha_{ij}=a_{ij}$
 and $\beta_k=b_k$) apply.
\cref{ld def} implies that PBW Conditions~(1) and~(3) hold. 
PBW Condition~(2)  is equivalent to the equation in Lemma~\ref{kappa invariance}; we examine the coefficient of 
each $g(i\ j\ k)$
and find that
 \[
 ^g
 \kappa_{(i\ j\ k)}(v_i, v_j)
 -\kappa_{(i\ j\ k)}(v_i, v_j)
 =(\ ^ga_{ijk}+c-a_{123})
 -(a_{ijk}+c-a_{123})
 =\ ^ga_{ijk}-a_{ijk}
 \]
 as desired. The coefficients of $g(i\ k\ j)$ are likewise equal in this equation
 and every other term trivially vanishes, giving 
 PBW Condition~(2).

PBW Condition~(4) is trivial except for
  $g=(i\ j\ k)$ and $v_i, v_j$ and $v_k$. By  \cref{kappa def implies equalities},
  \[
  \begin{aligned}
  &\kappa_{(i\ j\ k)}(v_i- v_j)(^gv_k, v_k)+\kappa_{(i\ j\ k)}(v_j- v_k)(^gv_i, v_i)+\kappa_{(i\ j\ k)}(v_k- v_i)(^gv_j, v_j)\\
  &\ \ 
  =\kappa_{(i\ j\ k)}(v_i, v_j)(v_i- v_k)+\kappa_{(i\ j\ k)}(v_j, v_k)(v_j- v_i)+\kappa_{(i\ j\ k)}(v_k, v_i)(v_k- v_j)=0.\\
  \end{aligned}\]

  To verify PBW Condition~(5) with $v_i, v_j$, and $v_k$, it suffices to consider 
  the coefficients of
  three group elements, 
  namely, $(i\ k\ j\ m)$, $(i\ k)$, and $(i\ j\ k)$ with indices
  distinct,
  as other terms all vanish.
  The coefficient of $(i\ k\ j\ m)$ 
  vanishes: one need only
  expand
  \[
  \begin{aligned}
 \kappa&_{(i\ j\ m)}(v_i, v_j)(a_{ki}-a_{kj})
  -\kappa_{(i\ k\ j)}(v_i, v_k)(a_{jm}-a_{im})\\
   &\ \ \ -\kappa_{(j\ k\ m)}(v_j, v_k)(a_{im}-a_{ik})
  -\kappa_{(i\ m\ k)}(v_k, v_i)(a_{jk}-a_{jm})\\
  &=
  (c-
  a_{123}+a_{ijm})
  (a_{ki}-a_{kj})
  -(c+-a_{123}+a_{ikj})(a_{jm}-a_{im})\\
  &\ \ \ \ -(c-a_{123}+a_{jkm})(a_{im}-a_{ik})
   -(c-a_{123}
   +a_{imk})(a_{jk}-a_{jm})
   \\
  &=a_{ki}(a_{ijm}-a_{jkm})+a_{kj}(a_{imk}-a_{ijm})+a_{jm}(a_{imk}-a_{ikj})+a_{mi}(a_{jkm}-a_{ikj})=0.
  \end{aligned}
  \]
  The coefficient of $(i\ k)$ is just
  \[\kappa_{(i\ j\ k)}(v_i, v_j)
  \big((a_{ij}-a_{jk})+(a_{ji}-a_{kj})-(a_{kj}-a_{ji})-(a_{jk}-a_{ij})\big)
  =0.\]
  The coefficient of $(i\ j\ k)$
  vanishes as well
  by  Lemma~\ref{gg fixed}:
  $$
  \kappa_{(i\ j\ k)}(v_i, v_j)
  \cdot
  \lambda_{(i\ j\ k)}\big((i\ j\ k), v_i+ v_j+ v_k\big)
  =
  0.
  $$
 Hence $\cH_{\mu}$ satisfies the PBW property by
 \cref{PBWconditions}.

\end{proof}

\begin{cor}
For $n>2$, the Drinfeld Hecke algebras for $\s_n$ acting on $\FF^n$ constitute a family defined by $\tfrac{1}{2}(n^2+n)$ parameters 
in $\FF$.
\end{cor}
\section{Drinfeld Hecke algebras for 
the symmetric group in low dimensions}
\label{lowdim}

We now give the Drinfeld Hecke algebras more explicitly
for $G=\mathfrak{S}_n$ acting in dimensions $n\leq 3$
by permuting basis vectors $v_1,\ldots, v_n$
of $V\cong\FF^n$. They all arise from an invariant parameter
$\kappa$.

\subsection*{One dimension}
The Drinfeld Hecke algebras $\cH_{\ld, \kappa}$ 
for $n=1$ are all trivial:
Theorem \ref{PBWconditions} forces 
$\kappa\equiv 0$ and $\lambda \equiv 0$
and $$
\cH_{\ld, \kappa} 
= \cH_{0, 0
}  =
\FF[v_1, \ldots, v_n]\# G.
$$

\subsection*{Two dimensions}
The Drinfeld Hecke algebras $\cH_{\ld, \kappa}$  
for
$n=2$ constitute a $2$-parameter family
given by
$$
\begin{aligned}
\cH_{a,b}=
\FF\langle
v,g: v\in V
\rangle
/
\big(
&g^2-1,\ \ 
v_1v_2-v_2v_1, \\
&(1\ 2)v_1- \,  v_2(1\ 2) - a-b(1\ 2),\\
&(1\ 2)v_2- \,   v_1(1\ 2) + a+b(1\ 2)
\big),
\end{aligned}
{}_{}
$$
for arbitrary scalars $a, b$ 
in $\FF$.
This follows from
Theorem \ref{PBWconditions} which forces 
$\ld((1\ 2), v_1)=-\ld((1\ 2), v_2)$ and $\kappa(v_1, v_2)=0$ as $\kappa(v_1, v_2)= \kappa(v_2, v_1)$ with $\text{char}(\FF)\neq 2$.

\begin{remark}{\em Note that if
we were to allow char$(\field)=2=n$, we would find instead
a $4$-parameter
family of Drinfeld Hecke algebras: 
for arbitrary $a,b,c,d$ in $\FF$,
$$
\begin{aligned}
\cH_{a,b,c,d}
=\FF\langle
v,g: v\in V
\rangle
/
\big(
&g^2-1,\ \ 
v_1v_2-v_2v_1-c-d(1\ 2),\\
&(1\ 2)v_1- \,   v_2(1\ 2) - a-b(1\ 2),\\
&(1\ 2)v_2- \,  v_1(1\ 2) - a-b(1\ 2)
\big).
\end{aligned}
{}_{}
$$
}
\end{remark}

\subsection*{Three dimensions}
The Drinfeld Hecke  algebras $\cH_{\lambda, \kappa}$
for $n=3$
constitute a $6$-parameter family:
\begin{prop}
Let
$\lambda: \FF G\otimes V \rightarrow \FF G$
and $\kappa: V
\otimes V \rightarrow \FF G$
be linear parameter functions with $\kappa$ alternating.
The algebra $\cH_{\lambda,\kappa}$
 generated by
a basis $v_1, v_2, v_3$ of $V$ and $\FF G$ with relations 
$$
\begin{aligned}
v_iv_j-v_jv_i=
\kappa(v_i, v_j)
\quad\text{ and }\quad
gv_i-\ ^gv_ig=\lambda(g,v_i)
\ \ \text{ for}\  g\in G
\end{aligned}
$$
satisfies the PBW
property if and only if
there are scalars $a_1, a_2, a_3, b_1, b_2,c$ in $\field$
with
$$
\kappa(v_i, v_j)
=
c\big((i\ j\ k)-(i\ k\ j)\big)
\quad\text{ for all }
i,j,k\ \text{ distinct}
$$ 
and
$\lambda$ is defined
by
\begin{small}
\[\begin{aligned}
&\lambda((1\ 2), v_1)=a_1 +b_1(1\ 2)-(a_2+a_3)(1\ 3\ 2),
&&\lambda((2\ 3), v_1)=(a_1+a_3)\big((1\ 3\ 2)+(1\ 2\ 3)\big),\\
&\lambda((1\ 2), v_2)=-a_1-b_1(1\ 2)+(a_2+a_3)(1\ 2\ 3),
&&\lambda((2\ 3), v_2)=a_2+b_2(2\ 3)-(a_1+a_3)(1\ 3\ 2),\\
&\lambda((1\ 2), v_3)=(a_2+a_3)\big((1\ 3\ 2)-(1\ 2\ 3)\big),
&&\lambda((2\ 3), v_3)=1+(2\ 3)-(1\ 2\ 3),\\
 &\lambda((1\ 3), v_2)=(a_1+a_2)\big((1\ 3\ 2)-(1\ 2\ 3)\big),
 &&
  \lambda((1\ 3), v_1)=-a_3-b_3(1\ 3)+(a_1+a_2)(1\ 2\ 3),\\
   &\lambda((1\ 3), v_3)=a_3-(a_1+a_2)(1\ 3\ 2)+b_3(1\ 3),
\end{aligned}
\]
\end{small}
\begin{small}
\[
\begin{aligned}
&\lambda((1\ 2\ 3), v_1)=(a_1-a_2)(1\ 3)-(a_3+a_2)(2\ 3)+b_1(1\ 2\ 3),\ \ \ \     
&& \\
&
\lambda((1\ 2\ 3), v_2)=(a_2+a_3)(1\ 2)+(a_2-a_1)(1\ 3)+b_2(1\ 2\ 3), 
&&  \\
&\lambda((1\ 2\ 3), v_3)=(a_2+a_3)(2\ 3)-(a_2+a_3)(1\ 2)+b_3(1\ 2\ 3),
&&  \\
&\lambda((1\ 3\ 2), v_1)=(a_2-a_3)(1\ 2)+(a_1-a_3)(2\ 3)-b_3(1\ 3\ 2), \\
&   \lambda((1\ 3\ 2), v_2)=(a_3-a_1)(2\ 3)+(a_2-a_1)(1\ 3)-b_2(1\ 3\ 2),\\
&\lambda((1\ 3\ 2), v_3)=(a_1-a_2)(1\ 3)+(a_3-a_2)(1\ 2)-b_1(1\ 3\ 2)\, ,
\end{aligned}
\]
\end{small}
\hspace{-1.5ex}
where $b_3=-(b_1+b_2)$.
\end{prop}
\begin{proof}
\cref{classification}
implies that
$\kappa_{(i\ j\ k)}(v_i, v_j)$ is determined by 
$a_{ijk}-a_{123}$, but for distinct $i, j, k$, $a_{ijk}=a_{123}$
since $a_{ij}=-a_{ji}$, which implies that $\kappa$ is defined as indicated. The parameter $\lambda$ is defined as in \cref{classification} with $a_1=a_{12}, a_2=a_{23}$, and $a_3=a_{31}$. 
\end{proof}


\section{Commutativity up to an
invariant parameter}
\label{sec:}
We saw in \cref{nonmodular}
that Drinfeld Hecke algebras in the nonmodular setting all arise from a
parameter $\kappa$ which is $G$-{\em invariant}.
Again, let $G=\s_n$ act on $V\cong\FF^n$ 
by permuting basis
vectors $v_1, \ldots, v_n$ of $V$.
In \cref{lowdim}, we observed that
the Drinfeld Hecke algebras
in the modular setting in low dimension
($n=1,2,3$)
also all arise from a parameter
$\kappa$ which is $G$-invariant.
By~\cref{classification},
other Drinfeld Hecke algebras
$\cH_{\ld,\kappa}$
with more general parameters
$\lambda$ and $\kappa$ arise,
but here 
we investigate those
with $\kappa$ invariant.
We assume $n> 2$ in this section.
Note that by PBW Condition~(2), the invariance
of $\kappa$ forces
$$
0=
\ld\bigl(\ld(g,v),u\bigr)-\ld\bigl(\ld(g,u),v\bigr)
\quad\text{ for}\quad
u,v\in V,\ g\in G\, .
$$
Again, we take indices on $b$ modulo $n$.
\begin{cor} 
\label{kappainvariant}
An algebra $\cH_{\lambda, \kappa}$ satisfies the PBW property with $\kappa$ invariant
if 
there are
scalars $c$, $d$,
$a_{1i}$,  $b_j$ in $\FF$
for 
$1< i \leq n$,
$1\leq j <n$, 
with the $a_{1i}$ distinct, such that
$$
\begin{aligned}
\kappa(v_i, v_j) 
&= \
c \sum_{k\neq i, j}  (i\ j\ k)-(i\ k\ j)
&&\text{ for }\quad
1\leq i\neq j \leq n ,
\ \text{ and }
\\
\lambda(g, v_i) 
&=
\sum_{k=0}^{g(i)-i+n-1}b_{i+k}\, g\
+
\sum_{j\neq i} (a_{ij}-a_{g(i)g(j)})
\, g(i\ j)
&&\text{ for }\quad
1\leq i\leq n\, ,
\end{aligned}
$$
where $a_{ij}=\frac{d+a_{1i}a_{1j}}{a_{1i}-a_{1j}}$ for $i,j \neq 1$, $i\neq j$, $a_{1i}=-a_{i1}$,
and $b_n=-(b_1+\cdots +b_{n-1})$.
\end{cor}
\begin{proof}
Consider the ordered tuple
$\mu=(a_{ij},b_k, c: 1\leq i< j \leq n, 1\leq k <n)$ 
and let $\cH_{\mu}$
be the PBW algebra of \cref{classification}.
A calculation confirms that $a_{ijk}=a_{ij}a_{jk}+a_{jk}a_{ki} +
a_{ki}a_{ij}=d$ for all distinct $i,j,k$
and thus $a_{ijk}-a_{123}=0$,
implying that $\cH_{\mu}
=\cH_{\lambda,\kappa}$.
Thus $\cH_{\lambda,\kappa}$
satisfies the PBW property
and one may check the invariance of
$\kappa$ directly.
\end{proof}
\begin{prop}\label{golden rule}
The algebra $\cH_{\lambda,0}$ satisfies the PBW
property
for
$\lambda:\field G \otimes V\rightarrow
\field G$
defined up to scalar in $\FF$ by
$$
\lambda(g, v_i) = 
(g(i)-i)\, g
\quad\text{ for all } 
g\in G,
\ 1\leq i\leq n.
$$
\end{prop}
\begin{proof}
In \cref{classification}, set $a_{ij}=0$, $c=0$, and $b_k=1$
for all $1\leq i<j
\leq n$, $1\leq k<n$.
Then $\kappa$ vanishes and $\ld(g, v)$ is supported on $g$ with
\[\ld(g, v_i)
=\sum_{k=0}^{g(i)-i+n-1}b_i\, g
=(g(i)-i)\, g,\]
as $b_n=-(n-1)$. 
For a scalar multiple, just choose the same
constant for all
$b_k$, $k<n$.
\end{proof}
Note that the converse of
\cref{golden rule}
fails: There are examples
of PBW algebras 
$\cH_{\lambda, 0}$
with $\lambda$ taking other forms.

\begin{cor}
If $\cH_{\lambda, 0}$
satisfies the PBW property,
then so does
$\cH_{\lambda,\kappa}$
for
$$
\kappa:V \otimes V\rightarrow
\field G\quad
\text{ defined by}\quad
\kappa (v_i, v_j) = 
\sum_{k\neq i, j}  (i\ j\ k)-(i\ k\ j)
\quad\text{ for }\quad
i\neq j.
$$
\end{cor}
\begin{proof}
 By~\cref{classification},
 there is an ordered tuple 
 $\mu=(a_{ij}, b_k, c
 :1\leq i< j\leq n, 
 i\leq k <n)$
 so that $\cH_{\mu}=\cH_{\lambda,0}$.
 Observe that
 $c=a_{123}-a_{ijk}$
 for all distinct $i,j,k$.
 Set $c'=c+1$
 and consider the ordered tuple
  $\mu'=(a_{ij}, b_k, c'
 :1\leq i< j\leq n, 
 i\leq k <n)$.
 The algebra
 $\cH_{\mu'}$ 
  satisfies the PBW 
  property by~\cref{classification}.
Since  
  $c'-a_{123}+a_{ijk}=1$
for all distinct $i,j,k$,
$\cH_{\mu'}=\cH_{\lambda,\kappa}$
for $\kappa$ as given in the statement,
and $\cH_{\lambda,\kappa}$
satisfies the PBW property.
\end{proof}

We end by highlighting a
handy $2$-parameter family of Drinfeld Hecke algebras.
\begin{cor}
The algebra generated by $v$ in 
$V=\field^n$ and $\field G$
with relations
$$
g v_i  -\ ^g v_i\, g 
= a \big(g(i)-i\big)g
\ \text{ and }\ 
v_i v_j - v_j v_i = 
b \sum_{k\neq i, j}  
(i\ j\ k)-(i\ k\ j)
\quad\text{ for all } g\in G, i\neq j
$$ 
satisfies the PBW property
for any $a,b$ in $\FF$.
\end{cor}
\begin{proof}
Use \cref{classification} with $a_{ij}=0$ , $c=1$, and $b_k=1$ for all $1\leq i<j\leq n$, $1\leq k<n$.
\end{proof}

\vspace{2ex}

\section{Acknowledgements}
The authors thank the referee for a very
careful reading of the article
and many helpful suggestions.

\end{document}